\documentclass[11pt]{article}

\usepackage{amsmath}
\usepackage{amssymb}
\usepackage{pifont}
\usepackage{geometry}
\usepackage{tikz-cd}
\usepackage{mathtools}

\usepackage[T2A,T1]{fontenc}
\DeclareTextSymbolDefault{\textsection}{OMS}
\DeclareRobustCommand{\Che}{
	\mathord{\text{\usefont{T2A}{cmr}{m}{n}\CYRCH}}
}

\usepackage{hyperref}
\hypersetup{
	colorlinks=true,
	linkcolor=blue,
	filecolor=magenta,      
	urlcolor=cyan,
	pdftitle={Concave transforms and volumes of adelic line bundles},
	pdfauthor={Yuxing Ye},
	pdfpagemode=FullScreen,
}
\usepackage{amsthm}
\theoremstyle{plain}
\newtheorem{theorem}{Theorem}[section]
\newtheorem{proposition}[theorem]{Proposition}
\newtheorem{lemma}[theorem]{Lemma}
\newtheorem{corollary}[theorem]{Corollary}

\theoremstyle{definition}

\theoremstyle{remark}
\newtheorem{remark}[theorem]{Remark}

\newcommand{\al}{\alpha}
\newcommand{\be}{\beta}
\newcommand{\R}{\mathbb{R}}
\renewcommand{\C}{\mathbb{C}}
\newcommand{\Spec}{\mathrm{Spec}\ }

\newcommand{\codim}{\mathrm{codim}}

\newcommand{\Pic}{\mathrm{Pic}}

\renewcommand{\P}{\mathbb{P}}

\renewcommand{\d}{\mathrm{d}}

\newcommand{\idp}{\mathfrak{p}}

\newcommand{\shO}{\mathcal{O}}

\newcommand{\shL}{\mathcal{L}}

\newcommand{\shC}{\mathcal{C}}
\newcommand{\shP}{\mathcal{P}}
\newcommand{\shE}{\mathcal{E}}

\newcommand{\shJ}{\mathcal{J}}

\newcommand{\eps}{\varepsilon}
\newcommand{\vol}{\widehat{\mathrm{vol}}_{\hat{\chi}}}

\begin{document}
	
	\title{Concave transforms and volumes of adelic line bundles}
	\author{Yuxing Ye}
	\date{\today}
	\maketitle
	\tableofcontents
	
	\section{Introduction}
	
	\paragraph{} Let $X$ be a projective variety of dimension $n$ over a number field $K$, and $L$ be a line bundle over $X$. For an adelic extension $\overline{L}$ of $L$, its \textit{characteristic volume} $\vol(\overline{L})$ is an invariant that measures the positivity of $\overline{L}$. 
	
	\paragraph{} Given a regular rational point $x_0\in X(K)$ and a system of parameters $t=(t_1,\cdots,t_n)$ at $x_0$, we can associate a convex body $\Delta(L)=\Delta_t(L)\subset \R^n$, called the \textit{Okounkov body}. A concave transform $c[\overline{L}]$ was introduced by Yuan in [\hyperlink{Yua09}{Yua09}]\footnote{Yuan suggested replacing the definition of $c[\overline{L}]$ by its negative; this is the convention used in this paper.} as the sum $\sum_{v\in M_K}n_vc_v[\overline{L}]$ of Nystr{\"o}m's local Chebyshev transforms [\hyperlink{Nys14}{Nys14}]. Yuan showed that for $L$ big,
	$$\int_{\Delta(L)}c[\overline{L}](\al)\d\al\ge\frac{1}{(n+1)!}\vol(\overline{L}).\footnote{The denominator $d!$ in [\hyperlink{Yua09}{Yua09}, Thm. 1.2] should be replaced by $(d+1)!$.}$$
	He also showed that equality holds under a certain boundedness condition, which is satisfied for toric varieties with $x_0$ and $t$ invariant under torus action. He further conjectured that this condition—and hence equality—always holds when $L$ is big.
	
	\paragraph{} We will show in this paper, however, that the equality can fail for curves of genus $g\ge 1$. Instead, the difference between the integral and the characteristic volume is related to the N{\'e}ron-Tate height $\hat{h}(L-dx_0)$ ($d=\deg L$) by the identity
	$$\int_{\Delta(L)}c[\overline{L}](\al)\d\al-\frac{1}{2}\vol(\overline{L})= [K:\mathbb{Q}]\hat{h}(L-dx_0).$$
	In particular, the difference depends only on the underlying line bundle $L$, not on the adelic metric.
	
	\subsection{Okounkov bodies and Yuan's concave transforms}
	
	\paragraph{} For $X$ and $\overline{L}$ as above, we have a map $\nu_t:H^0(X,mL)\setminus\{0\} \to \mathbb{N}^n$, defined as follows. For $0\neq s\in H^0(X,mL)$, take any local generator $s_0$ of $L$ at $x_0$ and consider the power series expansion
	$$s_0^{-m}s=\sum_{\al\in\mathbb{N}^n}a_\al t^\al.$$
	Define $\nu_t(s)$ as the smallest $\al$ in lexicographic order such that $a_\al\neq 0$. The \textit{Okounkov body} $\Delta_t(L)$ is defined as the closure of
	$$\bigcup_{m\ge 1}\frac{1}{m}\nu_t(H^0(X,mL)\setminus\{0\})\subset\R^n.$$
	
	\paragraph{} For $v\in M_K$ and $\al\in\nu_t(H^0(X,mL)\setminus\{0\})$, the \textit{local discrete Chebyshev transform} is defined as
	$$F_v[m\overline{L}](\al)=-\inf_{s\in H^0(X,mL)(\al)}\log\|s\|_{v,\sup},$$
	where
	$$H^0(X,mL)(\al)=\{s\in H^0(X,mL):s_0^{-m}s=t^\al+\text{higher-order  terms}\}.$$
	
	\paragraph{} For $\al$ in the interior of $\Delta_t(L)$, the \textit{local Chebyshev transform} is defined as the limit
	$$c_v[\overline{L}](\al)=\lim_{k\to\infty}\frac{1}{m_k}F_v[m_k\overline{L}](m_k\al_k).$$
	Here, $m_k\to\infty,\ m_k\al_k\in\nu_t(H^0(X,m_kL)\setminus\{0\})$, and $\al_k\to\al$. It exists and is independent of the choice of the sequence $\{(m_k,\al_k)\}_k$ by Proposition 1.1 of [\hyperlink{Yua09}{Yua09}]. 
	
	\paragraph{} The \textit{global discrete Yuan transform} and \textit{global Yuan transform}
	$$F[m\overline{L}]:\nu_t(H^0(X,mL)\setminus\{0\})\to\R,$$ $$c[\overline{L}]:\Delta_t(L)^\circ\to\R$$
	are defined respectively as the sums $$F[m\overline{L}](\al)=\sum_{v\in M_K}n_vF_v[m\overline{L}](\al),$$
	$$c[\overline{L}](\al)=\sum_{v\in M_K}n_vc_v[\overline{L}](\al),$$
	which are essentially finite by Lemma 2.4 (3a) and (3b) of [\hyperlink{Yua09}{Yua09}]. Here, $n_v=2$ if $K_v=\C$ and $n_v=1$ otherwise.
	
	\paragraph{} An alternative definition of $F[m\overline{L}](\al)$ is given by Boucksom and Chen in [\hyperlink{BC11}{BC11}] as the adelic degree\footnote{The adelic degree used in our paper is $[K:\mathbb{Q}]$ times the adelic degree in [\hyperlink{BC11}{BC11}].} of a $1$-dimensional adelically normed vector space
	$$\{s\in H^0(X,mL):s=0\ \mathrm{or}\ \nu_t(s)\ge\al\}/\{s\in H^0(X,mL):s=0\ \mathrm{or}\ \nu_t(s)>\al\}.$$
	Here, the inequality signs are with respect to lexicographic order.
	
	\begin{proposition}[{[\hyperlink{Yua09}{Yua09}, Prop.~2.5]}]\label{prop:chebyshev-continuity}
		The function $c[\overline{L}]$ is concave and continuous on $\Delta_t(L)^\circ$, and for any sequence $\{(m_k,\al_k)\}_k$ as above, we have
		$$c[\overline{L}](\al)=\lim_{k\to\infty}\frac{1}{m_k}F[m_k\overline{L}](m_k\al_k).$$
	\end{proposition}
	
	\paragraph{} Yuan's concave transform is closely related to several other constructions. The details are given in Appendix~\ref{app:comparison}.
	
	\paragraph{} From now on, we will usually omit $t$ from subscripts for brevity. Furthermore, we denote by $\nu(M)$ the image of $\nu_t:H^0(X,M)\setminus\{0\} \to \mathbb{N}^n$ for any line bundle $M$.
	
	\subsection{Characteristic volume}
	
	\paragraph{} For an adelic line bundle $\overline{L}$ over a projective variety $X$ of dimension $n$ with $L$ big, we define
	$$\hat{\chi}(X,\overline{L})=\log\frac{\mathrm{vol}(\mathbb{B}(\overline{L}))}{\mathrm{vol}\left(H^0(X,L)_{\mathbb{A}_K}/H^0(X,L)\right)},$$ 
	where $\mathbb{B}(\overline{L})=\prod_{v\in M_K}B_v(\overline{L})$ is the product of unit balls
	$$B_v(\overline{L})=\{s\in H^0(X,\overline{L})_{K_v}:\|s\|_{v,\sup}\le 1\}.$$
	The \textit{characteristic volume} of $\overline{L}$ is defined as
	$$\vol(\overline{L})=\limsup_{m\to\infty}\frac{(n+1)!}{m^{n+1}}\hat{\chi}(X,m\overline{L}).$$
	
	\paragraph{} Set $$D(\overline{L})=\int_{\Delta(L)}c[\overline{L}](\al)\d\al-\frac{1}{(n+1)!}\vol(\overline{L}).$$
	We will show in \S\ref{sec:hd-reduction} that $D(\overline{L})$ depends only on the underlying line bundle $L$. That is, for another adelic line bundle $\overline{L}'$ on $X$ with the same underlying line bundle $L$, we have $D(\overline{L}')=D(\overline{L})$.
	
	\subsection{The case of curves}
	
	\paragraph{} Throughout \S\S\ref{sec:first-inequality} and \ref{sec:second-inequality}, we fix a number field $K$ and a curve $C$ over $K$ of genus $g\ge 1$, with a point $x_0\in C(K)$ and a parameter $t$ at $x_0$ given. Let $\overline{L}$ be an adelic line bundle such that $\deg L=d>0$, with a given local generator $s_0$ at $x_0$ for convenience. The Riemann–Roch theorem implies that $\Delta(L)=[0,d]$. 
	
	\paragraph{} Denote by $\hat{h}:\Pic^0(C)\to \R_{\ge 0}$ the N{\'e}ron-Tate height on $\Pic^0(C)$, normalized as in \S\ref{sec:notation}. Our main result for curves is:
	
	\begin{theorem}[Corollary~\ref{cor:first-inequality} and Proposition~\ref{prop:second-inequality}]\label{thm:main}
		We have
		$$D(\overline{L})=[K:\mathbb{Q}]\hat{h}(L-dx_0).$$
		In particular, $D(\overline{L})=0$ if and only if $L=dx_0$ in $\Pic(C)_{\mathbb{Q}}$.
	\end{theorem}
	
	\paragraph{} We remark that as shown in [\hyperlink{Yua09}{Yua09}], the integral and the characteristic volume agree for ample line bundles on the projective line.
	
	\paragraph{} Our proof will proceed as follows. In \S\S\ref{sec:hd-reduction} and \ref{sec:first-inequality}, we prove the independence of metric and the first inequality $D(\overline{L})\ge[K:\mathbb{Q}]\hat{h}(L-dx_0)$ by reducing to the sum of top discrete Yuan transforms. In \S\ref{sec:second-inequality}, we prove the reverse inequality $D(\overline{L})\le[K:\mathbb{Q}]\hat{h}(L-dx_0)$ by exploiting the specific admissible adelic extension. This completes the proof of the identity.
	
	\subsection{The case of higher dimension}
	
	\paragraph{} In \S\ref{sec:higher dim}, we consider a projective variety $X$ of dimension $n\ge 2$ over a number field $K$. Fix a regular point $x_0\in X(K)$ and a system of parameters $t=(t_1,\cdots,t_n)$ at $x_0$. Assume that $\overline{L}$ is an adelic line bundle on $X$, that its underlying line bundle $L$ is big, and that a local generator $s_0$ of $L$ at $x_0$ is chosen. Our main result in this case is:
	
	\begin{theorem}[Theorem~\ref{thm:ample-ineq}]\label{thm:hd-main}
		Assume that $X$ is smooth, $L$ is ample, and that there is a flag of smooth subvarieties
		$$X=X_0\supset X_1\supset\cdots\supset X_{n-1}=C\supset X_n=\{x_0\}$$
		for which $\codim(X_{i+1},X_i)=1$, $t_i$ is a local equation of $X_i$ in $X_{i-1}$ at $x_0$ for any $1\le i\le n$, and $C$ is a curve of genus $g\ge 1$.\\
		
		Take positive real numbers $\lambda_i\ (0\le i\le n-2)$ such that $\lambda_iL|_{X_i}-\shO_{X_i}(X_{i+1})$ is nef. Set
		$$v_i=\shO_{X_i}(X_{i+1})|_C-\deg\left(\shO_{X_i}(X_{i+1})|_C\right)x_0\in\Pic^0(C),$$
		and $u=L|_C-\deg(L|_C)x_0\in \Pic^0(C)$. Then,
		$$D(\overline{L})\ge\frac{1}{\lambda_0\cdots\lambda_{n-2}}[K:\mathbb{Q}]\int_{\Delta_{n-1}}\hat{h}_C\left(u-\sum_{i=0}^{n-2}\frac{\al_i}{\lambda_i}v_i\right) \d\al_0\cdots\d\al_{n-2}.$$
		Here, $\Delta_{n-1}$ is the standard simplex in $\R^{n-1}$, and $\hat{h}_C$ is extended to a quadratic form on $\Pic^0(C)_\R$. In particular, $D(\overline{L})$ is strictly positive if at least one of $u, v_0,\cdots,v_{n-2}$ is non-torsion, since the N{\'e}ron-Tate form is positive definite on $\Pic^0(C)_\R$.
	\end{theorem}
	
	\paragraph{} To prove this lower bound, we reduce to the case of curves by a jet-norm inequality.
	
	\subsection{Notations and terminology}\label{sec:notation}
	
	\paragraph{} A \textit{number field} $K$ is a finite extension of $\mathbb{Q}$. Denote by $O_K$ the ring of algebraic integers in $K$, $M_{K}$ the set of places of $K$, and $M_{K,f}\ (\mathrm{resp.}\ M_{K,\infty})$ the set of nonarchimedean (resp. archimedean) places.
	
	\paragraph{} By a \textit{variety} over a field $k$, we mean an integral, separated, finite-type $k$-scheme. By a \textit{curve} $C$ over a field $k$, we mean a geometrically integral, smooth, projective $k$-scheme of dimension $1$. 
	
	\paragraph{} For the Jacobian variety $J$ of a curve $C$ and $x_0\in C(K)$, the \textit{Abel-Jacobi map}
	$$i_{x_0}:C\to J$$
	sends $x$ to $x-x_0$ on $\overline{K}$-points. The \textit{theta divisor} $\theta_{x_0}$ is the image of the $(g-1)$-fold Abel-Jacobi map $i_{x_0}:C^{(g-1)}\to J$ and $\Theta$ is defined by $\Theta=\theta_{x_0}+[-1]^*\theta_{x_0}$. The \textit{N{\'e}ron-Tate height} $\hat{h}=\hat{h}_C:\Pic^0(C)\to\R_{\ge 0}$ is half of the canonical height on $J(K)$ associated to $\Theta$.
	
	\paragraph{} By a \textit{line bundle} on a scheme, we mean an invertible sheaf. In the case of curves, if $x$ is a closed point, then $x$ stands for the line bundle $\shO(x)$ when there is no risk of confusion. We will use additive notation for tensor products, that is, $aL-bM=L^{\otimes a}\otimes M^{\otimes(-b)}$. Our definitions for hermitian and adelic line bundles follow [\hyperlink{YG25}{YG25}] and [\hyperlink{YZ26}{YZ26}].
	
	\paragraph{} For a flat projective morphism $\pi:X\to B$ whose fibers are curves, by an \textit{effective relative divisor} on $X$, we mean an effective Cartier divisor which defines a closed subscheme flat over $B$. For a projective arithmetic variety (i.e. integral scheme flat and projective over $\Spec \mathbb{Z}$), a prime divisor on $X$ is \textit{horizontal} (resp. \textit{vertical}) if it is dominant (resp. not dominant) over $\Spec\mathbb{Z}$.
	
	\subsection*{Statement on AI use}
	
	\paragraph{} The paper was written by the author. OpenAI’s GPT-5.6 Sol assisted with the work in Section 4 in the following ways:
	
	\begin{enumerate}
		\item Suggesting that the $g+1$ largest values in $\nu(mL)$ be considered, rather than only the largest value as in the author’s original approach. The latter approach would yield the weaker lower bound
		$$D(\overline{L})\ge \frac{[K:\mathbb{Q}]}{(g-1)\lfloor g/2\rfloor+g}\hat{h}(L-dx_0).$$
		\item Completing the computation of the determinant line bundle in Lemma~\ref{lem:picard-determinant}.
	\end{enumerate}
	
	\paragraph{} All mathematical statements, proofs, computations, and verifications contributed by ChatGPT (GPT-5.6 Sol, OpenAI) in the work described above were subsequently examined, rigorously checked, and rewritten by the author, who takes full responsibility for the results.
	
	\subsection*{Acknowledgments}
	
	\paragraph{} The author would like to express his sincere gratitude to Professor Xinyi Yuan for raising the problem studied in this paper and for his valuable guidance and advice. The author also thanks doctoral student Jiawei Yu for carefully checking the correctness of the manuscript. 
	
	\section{Bounds on the sup norm}
	
	\paragraph{} In this section, we present some bounds that will be used later. We first prove a relative variant of Lemma 2.4 in [\hyperlink{Yua09}{Yua09}]. Its geometric intuition is that a Green function of an effective divisor without vertical singularities has a uniformly bounded infimum over the fibers. Let $K$ be a number field.
	
	\begin{lemma}\label{lem:sup-norm-bound}
		Let $X$ be a projective variety over $K$, $C$ be a curve over $K$, $E$ be an effective relative divisor on $X\times C$ over $X$, and $\overline{E}$ be an adelic extension of $E$. Then for any $\delta>0$, there exist a constant $c>0$ and a finite set of places $S\subset M_K$, containing all archimedean places, such that for any $x\in X(\overline{K})$ of degree $\deg x\le\delta$,
		\begin{enumerate}
			\item For any $v\in M_{k(x)}$ lying over $M_K\setminus S$, $\|1\|_{\shO(\overline{E})|_{C_x},v,\sup}=1$.
			\item For any $v\in M_{k(x)}$ lying over $S$, $\left|\log \|1\|_{\shO(\overline{E})|_{C_x},v,\sup}\right|<c$.
		\end{enumerate}
		Here, for an adelic line bundle $\overline{M}$ on a projective variety $Y$ over a field $K'$, denote by
		\[
		\|s\|_{\overline{M},v,\sup}:=\sup_{y\in Y(\C_v)}|s(y)|_{\overline{M},v}.
		\qquad (s\in H^0(Y,M),\ v\in M_{K'}).
		\]
	\end{lemma}
	
	\begin{proof}
		We first prove (1). By the coherence condition, there exist integral projective $O_K$-models $\mathcal{X}, \mathcal{C}$ of $X,C$, and an effective divisor $\shE$ on $\mathcal{X}\times_{O_K}\mathcal{C}$, such that $\overline{E}$ is induced by $(\mathcal{X}\times_{O_K}\mathcal{C},\mathcal{E})$ at all places outside a finite set $S_1\supset M_{K,\infty}$.\\
		\par Take $A=(\deg E+1)$ points $y_1, \cdots, y_A\in C(\overline{K})$ with distinct images in $C$. Denote by $\overline{y}_i$ their closures in $\mathcal{C}$. Then there is a finite set $S_2\subset M_{K,f}$ such that for any maximal ideal $\mathfrak{p}$ of $O_K$ whose corresponding place is not in $S_2$, $\overline{y}_i\cap(k(\idp)\times_{O_K}\mathcal{C})$ are disjoint. By openness of the flat locus and the fact that $E$ is a relative divisor, there is a finite set $S_3\subset M_{K,f}$ such that for any $v\in M_{K,f}\setminus S_3$ and closed point $P\in \mathcal{X}$ lying over $v$, $k(P)\times_{O_K}\mathcal{C}$ is not contained in $\mathrm{supp}(\shE)$, and $\mathcal{E}|_{k(P)\times_{O_K}\mathcal{C}}$ has degree equal to $\deg E$. Let $S=S_1\cup S_2\cup S_3$.\\
		\par For any $x\in X(\overline{K})$, denote by $\mathcal{E}_x$ and $\mathcal{Y}_i$ the pullbacks of $\shE$ and $\overline{y}_i$ to $O_{k(x)}\times_{O_K}\mathcal{C}$, where $O_{k(x)}\to \mathcal{X}$ extends $x\to\mathcal{X}$. Then for $v\in M_K\setminus S$ and $\idp\in\Spec O_{k(x)}$ lying over $v$, $\mathcal{Y}_i\cap(k(\idp)\times_{O_K}\mathcal{C})$ are disjoint, and $\deg \shE_x|_{k(\idp)\times_{O_K}\mathcal{C}}=\deg E< A$. Therefore, there exists $i$ such that $\shE_x$ and $\mathcal{Y}_i$ are disjoint on $k(\idp)\times_{O_K}\mathcal{C}$. By evaluating the metric at $y_i$, $\|1\|_{\shO(\overline{E})|_{C_x},v,\sup}\ge 1$. By effectiveness, $\|1\|_{\shO(\overline{E})|_{C_x},v,\sup}\le 1$. This completes the proof of (1).\\
		\par Next we show (2). By a familiar result on local fields, we can find a finite Galois extension $F/K_v$ such that any finite extension of $K_v$ of degree at most $\delta$ admits an embedding into $F$. Therefore, $x$ lifts to a point $x'\in X(F)$. The space $X(F)$, endowed with its analytic topology, is compact. By continuity of the metric, the map $x'\mapsto \log \|1\|_{\shO(\overline{E})|_{C_{x'}},v,\sup}$ is lower semicontinuous and is therefore bounded below.\\
		\par For the upper bound, if $v$ is non-archimedean, take a normal model $(\mathcal{Z},\shE)$ over $O_K$ of $(X\times C, E)$ such that $\|\cdot\|_v\le \mu\|\cdot\|_{\mathrm{mod},v}$ for some $\mu>0$. Passing to a multiple, we may assume that $\shE$ is an integral divisor. Its horizontal components must have positive multiplicity. Take $B\ge 0$ such that the multiplicity of any vertical component of $-\shE$ is not greater than $B$. Then, $\|1\|_{\mathrm{mod},v}\le (\# k(v))^B$, so the map $x'\mapsto \log \|1\|_{\shO(\overline{E})|_{C_{x'}},v,\sup}$ is bounded above. If $v$ is archimedean, then the uniform upper bound follows from the fact that $(X\times C)(\C)$ is compact in the analytic topology. This completes the proof.
	\end{proof}
	
	\begin{corollary}\label{cor:sup-norm-bound}
		Let $X$ be a projective variety over $K$, $C$ be a curve over $K$, $E_1, E_2$ be two effective relative divisors on $X\times C$ over $X$, and $\overline{E}_i$ be adelic extensions of $E_i$. Then for any $\delta>0$, there exist a constant $c>0$ and a finite set of places $S\subset M_K$, containing all archimedean places, such that for any $x\in X(\overline{K})$ of degree $\deg x\le\delta$ and $m\in \mathbb{Z}_{>0}$,
		\begin{enumerate}
			\item For any $v\in M_{k(x)}$ lying over $M_K\setminus S$, $\|1\|_{\shO(m\overline{E}_1+\overline{E}_2)|_{C_x},v,\sup}=1$.
			\item For any $v\in M_{k(x)}$ lying over $S$, $\left|\log \|1\|_{\shO(m\overline{E}_1+\overline{E}_2)|_{C_x},v,\sup}\right|<cm$.
		\end{enumerate}
	\end{corollary}
	
	\begin{proof}
		Combine Lemma~\ref{lem:sup-norm-bound} applied to $E_1, E_2, (E_1+E_2)$ with the inequalities
		$$\log \|1\|_{\shO(m\overline{E}_1+\overline{E}_2)|_{C_x},v,\sup}\le m\log \|1\|_{\shO(\overline{E}_1)|_{C_x},v,\sup}+\log \|1\|_{\shO(\overline{E}_2)|_{C_x},v,\sup};$$
		$$\log \|1\|_{\shO(m(\overline{E}_1+\overline{E}_2))|_{C_x},v,\sup}\le \log \|1\|_{\shO(m\overline{E}_1+\overline{E}_2)|_{C_x},v,\sup}+(m-1)\log \|1\|_{\shO(\overline{E}_2)|_{C_x},v,\sup}.$$
		We obtain the result.
	\end{proof}
	
	\paragraph{} For completeness, we include the following basic fact.
	
	\begin{lemma}\label{lem:trivial-underlying-divisor}
		Let $X$ be a projective variety over $K$, $\overline{E}$ be an adelic divisor on $X$ with underlying divisor $E=0$. Then,
		\begin{enumerate}
			\item $0<\inf_{Q\in X(\C_v)}\|1(Q)\|_{\shO(\overline{E}),v}\le\sup_{Q\in X(\C_v)}\|1(Q)\|_{\shO(\overline{E}),v}<\infty$ for every $v\in M_K$.
			\item For all but finitely many places, both the supremum and infimum above are $1$.
		\end{enumerate}
	\end{lemma}
	
	\begin{proof}
		It suffices to prove the inequalities involving the supremum, since
		$$\inf_{Q\in X(\C_v)}\|1(Q)\|_{\shO(\overline{E}),v}=(\sup_{Q\in X(\C_v)}\|1(Q)\|_{\shO(-\overline{E}),v})^{-1}.$$
		By the coherence condition, for all but finitely many places, the norm of the section $1$ is $1$ everywhere. For the remaining places, use the same argument as the upper-bound part in Lemma~\ref{lem:sup-norm-bound}.
	\end{proof}
	
	\section{Boundary effect and independence of the metric}\label{sec:hd-reduction}
	
	\paragraph{} We first prove that the contribution to $D(\overline{L})$ comes from discrete transforms near the boundary, in the sense of the following limit:
	
	\begin{lemma}\label{lem:hd-top-discrete-transforms}
		We have
		
		\begin{equation}\notag
			\begin{split}
				D(\overline{L}) &\ =\lim_{\eps\to 0^+}\lim_{m\to\infty}\frac{1}{m^{n+1}}\sum_{\al\in\nu(mL),\mathrm{dist}\left(\al/m,\partial(\Delta(L))\right)<\eps}\left(-F[m\overline{L}](\al)\right)\\
				&\ =\lim_{\eps\to 0^+}\lim_{m\to\infty}\frac{1}{m^{n+1}}\sum_{\al\in\nu(mL),\mathrm{dist}\left(\al/m,\partial(\Delta(L))\right)<\eps}\max\{-F[m\overline{L}](\al),0\}.
			\end{split}
		\end{equation}
	\end{lemma}
	
	\begin{proof}
		Denote by
		$$\Delta_\eps(L)=\{\al\in\Delta(L):\mathrm{dist}(\al,\partial\Delta(L))\ge\eps\}.$$
		It suffices to prove the first equality since the second would then follow from the uniform upper bound [\hyperlink{Yua09}{Yua09}, Lem. 2.4] and a boundary count
		$$\frac{\#\{\alpha\in\nu(mL):\mathrm{dist}(\alpha/m,\partial\Delta)<\eps\}}{m^n}\le\mathrm{vol}(\Delta(L)\setminus\Delta_\eps(L))+o(1).$$
		By \S2.4 of [\hyperlink{Yua09}{Yua09}], we have
		$$\frac{1}{(n+1)!}\vol(\overline{L})=\lim_{m\to\infty}\frac{1}{m^{n+1}}\sum_{\al\in\nu(mL)}F[m\overline{L}](\al).$$
		Therefore, it suffices to prove the equalities
		$$\lim_{\eps\to 0^+}\int_{\al\in\Delta(L)\setminus\Delta_\eps(L)}c[\overline{L}](\al)\d\al=0;$$
		$$\lim_{m\to\infty}\frac{1}{m^n}\sum_{\al\in\nu(mL),\al/m\in\Delta_\eps(L)}\frac{1}{m}F[m\overline{L}](\al)=\int_{\al\in\Delta_\eps(L)}c[\overline{L}](\al )\d\al.$$
		The first holds since $c[\overline{L}]$ is integrable. To show the second equality, apply the dominated convergence theorem to the sequence of step functions 
		$$\sum_{\al\in\nu(mL),\al/m\in\Delta_\eps(L)}\frac{1}{m}F[m\overline{L}](\al)1_{\prod_{k=1}^n[\frac{\al_k-1}{m},\frac{\al_k}{m})}$$
		which converges almost everywhere to $1_{\Delta_\eps(L)}c[\overline{L}]$. We apply a reformulation of Khovanskii's result [\hyperlink{Yua09}{Yua09}, Cor. 2.2]. It says that for every convex body $D\subset \Delta(L)^\circ$ and all sufficiently large $m$, we have
		$$D\cap\frac{1}{m}\nu(mL)=D\cap\frac{1}{m}\mathbb{Z}^n.$$
		Now take a sufficiently large integer $N_0>\frac{2\sqrt{n}}{\eps}$. Then, for any $N\ge N_0$, we have
		$$N\Delta_{\eps/2}(L)\cap\mathbb{Z}^n\subset\nu(NL).$$
		Any $m\ge N_0$ can be expressed as a sum of integers $m_1,\cdots,m_j\in[N_0,2N_0)$, and any vector $\al\in m\Delta_\eps(L)\cap\mathbb{Z}^n$ can be expressed as a sum of vectors
		$$\al^{(i)}\in \left(\frac{m_i}{m}\al+[-1,1]^n\right)\cap\mathbb{Z}^n\subset m_i\Delta_{\eps/2}(L)\cap\mathbb{Z}^n\subset\nu(m_iL)$$
		where $1\le i\le j$. This can be achieved by choosing $\al^{(i)}$ inductively such that $$\left|\sum_{l=1}^i\al^{(l)}_k-\frac{\sum_{l=1}^im_l}{m}\al_k\right|\le\frac{1}{2}.$$
		Therefore,
		$$\frac{1}{m}F[m\overline{L}](\al)\ge\frac{1}{\sum_{i=1}^j m_i}\sum_{i=1}^j F[m_i\overline{L}](\al^{(i)})\ge\min_{N_0\le N< 2N_0,\al'\in\nu(NL)}\{\frac{1}{N}F[N\overline{L}](\al')\}.$$
		These quantities are also bounded above by $\sup_{\Delta_{\eps/2}(L)}c[\overline{L}]$, since
		$$c[\overline{L}](\frac{\al}{m})=\lim_{l\to\infty}\frac{1}{ml}F[lm\overline{L}](l\al)\ge\lim_{l\to\infty}\frac{1}{ml}\cdot lF[m\overline{L}](\al)=\frac{1}{m}F[m\overline{L}](\al).$$
		Thus the dominated convergence theorem applies.
	\end{proof}
	
	\begin{remark}\label{rmk:one-face}
		We present a variant of the preceding boundary effect, which involves only one boundary face.\\
		\par Assume $X=C\times\P_K^{n-1}\ (n\ge 1)$ and $x_0=(y_0,[1:0:\cdots:0])$. Set $t_i=\frac{X_i}{X_0}$ for $1\le i\le n-1$, and let $t_n$ be a parameter of $C$ at $y_0$. In this case, we have $L=M\boxtimes \shO(b)$, where $b\in\mathbb{Z}_{>0}$ and $M\in\Pic(C)$ has degree $d>0$. By the K{\"u}nneth formula,
		$$H^0(X,mL)=H^0(C,mM)\otimes H^0(\P_K^{n-1},\shO(bm)).$$
		Therefore, we have $$\nu(mL)=\nu_{(t_1,\cdots,t_{n-1})}(\shO(bm))\times\nu_{t_n}(mM),$$
		and the Okounkov body is of the form
		$$\Delta(L)=b\Delta_{n-1}\times [0,d].$$
		Note that $\nu_{(t_1,\cdots,t_{n-1})}(\shO(bm))=bm\Delta_{n-1}\cap\mathbb{Z}^{n-1}$ and $\nu_{t_n}(mM)\supset\{0,1,\cdots,md-2g\}$. Therefore, for any $\eps>0$ and all sufficiently large $m$, we may replace Khovanskii's result with the stronger inclusion
		$$(b\Delta_{n-1}\times [0,d-\eps])\cap\frac{1}{m}\nu(mL)=(b\Delta_{n-1}\times [0,d-\eps])\cap\frac{1}{m}\mathbb{Z}^n,$$
		and deduce the required uniform lower bound. The uniform upper bound follows from the concavity of $c[\overline{L}]$. Hence we obtain the boundary effect
		\begin{equation}\notag
			\begin{split}
				D(\overline{L}) &\ =\lim_{\eps\to 0^+}\lim_{m\to\infty}\frac{1}{m^{n+1}}\sum_{\al\in\nu(mL),\al_n/m>d-\eps}\left(-F[m\overline{L}](\al)\right)\\
				&\ =\lim_{\eps\to 0^+}\lim_{m\to\infty}\frac{1}{m^{n+1}}\sum_{\al\in\nu(mL),\al_n/m>d-\eps}\max\{-F[m\overline{L}](\al),0\}.
			\end{split}
		\end{equation}
	\end{remark}
	
	\paragraph{} We now deduce the independence of $D(\overline{L})$ from the choice of adelic metric:
	
	\begin{theorem}\label{thm:hd-metric-independence}
		If $\overline{L}'$ is another adelic line bundle on $X$ with the same underlying line bundle $L$, then $D(\overline{L}')=D(\overline{L})$.
	\end{theorem}
	
	\begin{proof}
		By Lemma~\ref{lem:trivial-underlying-divisor}, $\frac{1}{m}(F[m\overline{L}](\al)-F[m\overline{L}'](\al))$ is uniformly bounded, so applying Lemma~\ref{lem:hd-top-discrete-transforms} gives the proof.
	\end{proof}
	
	\section{The first inequality}\label{sec:first-inequality}
	
	\paragraph{} The following two sections are devoted to the proof of Theorem~\ref{thm:main}. In this section, we prove the first inequality $D(\overline{L})\ge[K:\mathbb{Q}]\hat{h}(L-dx_0)$ by computing the discrete transform at the $(g+1)$ largest values of $\nu(mL)$ in \S\ref{sec:highest-transforms} and \S\ref{sec:picard-determinant}, and applying the boundary effect in the previous section. For convenience (e.g. to apply Gromov's inequality), we assume from now on that the metrics of $L$ at archimedean places are smooth.
	
	\subsection{Sum of highest discrete transforms}\label{sec:highest-transforms}
	
	\paragraph{} We shall estimate the sum of discrete transforms at the $(g+1)$ largest values in $\nu(mL)$, which is associated with the degree of an adelically normed vector space. Set $e_0=e=2g$. Then, for $md>e$, Riemann–Roch gives
	$$\nu(mL)=\{0,1,\cdots,md-e,md-e_1,\cdots,md-e_g\},$$
	where $0\le e_g< \cdots< e_1<e_0=e$.
	
	\paragraph{} Let
	$$W_j=\{s\in H^0(C,mL):s=0\ \mathrm{or}\ \nu(s)\ge md-e_j\}$$
	for $0\le j\le g$ and $W_{g+1}=0$. Denote by $\overline{W}_j$ the $M_K$-normed vector space endowed with the $\sup$ norm at non-archimedean places and $L^2$ norm at archimedean places on $W_j$.
	
	\begin{proposition}\label{prop:filtration}
		\[
		\overline{W}_0\supset\overline{W}_1\supset\cdots\supset\overline{W}_{g+1}=\{0\}
		\]
		is a filtration with $1$-dimensional successive quotients, and for $0\le j\le g$,
		\[
		F[m\overline{L}](md-e_j)=\widehat{\mathrm{deg}}(\overline{W}_j/\overline{W}_{j+1})+O(\log m).
		\]
		As a consequence,
		\[
		\sum_{j=0}^{g}F[m\overline{L}](md-e_j)=\widehat{\mathrm{deg}}(\det \overline{W}_0)+O(\log m).
		\]
	\end{proposition}
	
	\begin{proof}
		The result follows from the alternative characterization by Boucksom-Chen and Gromov's inequality [\hyperlink{GS92}{GS92}, Lem. 30], which implies that the logarithms of $\sup$ and $L^2$ norms on $W_0$ at the archimedean places differ by $O(\log m)$. See also [\hyperlink{Yua09}{Yua09}, Prop. 2.9].
	\end{proof}
	
	\paragraph{} Now we endow $W_0$ with another norm. Denote by $J$ the Jacobian variety of $C$, with $\pi_C:J\times C\to C$ the projection. Let $\shP_C\in\Pic(J\times C)$ be the Poincar{\' e} line bundle rigidified along $J\times x_0$. It admits a natural adelic extension by a Tate limiting argument (see [\hyperlink{YZ26}{YZ26}, Thm. 6.1.1]). Denote this extension by $\overline{\shP}_C$. Let $L_0=L-dx_0$ correspond to $Q\in J(K)$. Choose any adelic extension $\overline{x}_0$ of the divisor $x_0$. Since the result is independent of the choice of adelic extension, we may assume that $\overline{L}_0=\overline{L}-d\shO(\overline{x}_0)=\overline{\shP}_C|_{Q\times C}$.
	
	\paragraph{} In this sense, $W_j$ can be identified with the space $H^0(C,mL_0+e_jx_0)$. Denote by $\overline{W}'_0$ the $M_K$-normed vector space endowed with the $\sup$ norm at non-archimedean places and $L^2$ norm at archimedean places with respect to $m\overline{L}_0+e\overline{x}_0$ on $W_0$.
	
	\begin{proposition}\label{prop:norm-comparison}
		There exists a constant $c>0$ and a finite set of places $S\subset M_K$, containing all archimedean places, such that for any $m\in \mathbb{Z}_{>0}$,
		\begin{enumerate}
			\item For any $v\in M_K\setminus S$, the $v$-norms on $\overline{W}_0$ and $\overline{W}'_0$ are identical.
			\item For any $v\in S$, the logarithms of $v$-norms on $\overline{W}_0$ and $\overline{W}_0'$ differ by at most $c(m+1)$.
		\end{enumerate}
	\end{proposition}
	
	\begin{proof}
		Denote by $C^e$ the $e$-fold product $C\times\cdots\times C$, $p_i:C^e\times C\to C$ the projection to the $i$-th factor of $C^e$, $p_C:C^e\times C\to C$ the projection to the last factor, and $p_{i,C}:C^e\times C\to C\times C$ the projection to the $i$-th and last factors. We have a line bundle
		$$M=\sum_{i=1}^ep_{i,C}^*\shO(\Delta)-\sum_{i=1}^ep_i^*\shO(x_0)-ep_C^*\shO(x_0)\in\Pic(C^e\times C)$$
		with a rigidification along $C^e\times x_0$. It corresponds to a morphism
		$$\phi=\sum_{i=1}^ei_{x_0}\circ p_i:C^e\to J$$
		which induces an isomorphism of rigidified line bundles $(\phi\times \mathrm{id})^*\shP_C\cong M$. Endow $M$ with an adelic metric such that the displayed isomorphism becomes an isometry $(\phi\times \mathrm{id})^*\overline{\shP}_C\cong \overline{M}$.\\
		\par For any $s\in H^0(C,mL_0+ex_0)\setminus\{0\}$, the zero divisor of $s$ corresponds to a point $R\in C^e(\overline{K})$ such that $\phi(R)=mQ$. Also, $s$ defines a section of
		$$m\overline{L}=(\overline{\shP}_C+md\pi_C^*\overline{x}_0)|_{mQ\times C}.$$
		Let $\tilde{s}$ be its pullback to $(\overline{M}+mdp_C^*\overline{x}_0)|_{R\times C}$. Set
		$$\overline{M}'=\overline{M}+\sum_{i=1}^ep_i^*\shO(\overline{x}_0)+ep_C^*\shO(\overline{x}_0).$$
		Then, for $v\in M_K$, the $v$-$\sup$ norms of $s\in W_0$ with respect to $\overline{W}_0$ and $\overline{W}_0'$ are
		$$\|\tilde{s}\|_{\overline{M}+mdp_C^*\overline{x}_0,v,\sup}=B_{s,v}\|1|_{R\times C}\|_{\overline{M}'+(md-e)p_C^*\overline{x}_0,v,\sup}$$
		and
		$$\|\tilde{s}\|_{\overline{M}+ep_C^*\overline{x}_0,v,\sup}=B_{s,v}\|1|_{R\times C}\|_{\overline{M}',v,\sup}$$
		respectively, where $B_{s,v}>0$ depends on the scalar multiple chosen for $s$. For archimedean places, we may switch between $\sup$ and $L^2$ norms by Gromov's inequality. Since
		$$\deg R\le e!,$$
		the result follows from Corollary~\ref{cor:sup-norm-bound}.
	\end{proof}
	
	\paragraph{} Combining Propositions~\ref{prop:filtration} and~\ref{prop:norm-comparison}, we obtain the following identity.
	
	\begin{corollary}\label{cor:top-transform-sum}
		In the notation above, we have
		\[
		\sum_{j=0}^{g}F[m\overline{L}](md-e_j)=\widehat{\mathrm{deg}}(\det \overline{W}_0)+O(m).
		\]
	\end{corollary}
	
	\subsection{Determinant of the Picard bundle}\label{sec:picard-determinant}
	
	\paragraph{} We now relate the adelic degree in the previous subsection to the N{\' e}ron-Tate height. The key construction is a global version of the determinant space.
	
	\paragraph{} Let $p_J$ and $p_C$ be the projections from $J\times C$ to its factors. For $N\in\mathbb{Z}$, denote by
	$$\shE_N=(p_J)_*(\shP_C\otimes p_C^*\shO(Nx_0)).$$
	For $N\ge 2g-1$, by the theorem on cohomology and base change, $\shE_N$ is a locally free sheaf of rank $N+1-g$, called the \textit{Picard bundle}. In this case, for $Q'\in J(\overline{K})$ corresponding to $L_0'\in \Pic^0(C_{k(Q')})$, the fiber of $\shE_N$ at $Q'$ is naturally identified with $H^0(C_{k(Q')},L_0'+Nx_0)$. 
	
	\paragraph{} Since $\shP_C$ admits a natural adelic extension $\overline{\shP}_C$, we may take an adelic extension $\overline{x}_0$ of $x_0$ and endow $\shE_N$ with a continuous $M_K$-metric which is the $\sup$ norm at non-archimedean places and $L^2$ norm at archimedean places, with respect to the identification with $H^0(C_{k(Q')},L_0'+Nx_0)$. Denote this metrized vector bundle by $\overline{\shE}_N$.
	
	\paragraph{} Now let $N=2g$ and $Q_0=mQ$, where $Q$ corresponds to $L_0=L-dx_0\in\Pic^0(C)$ as before. We may rephrase Corollary~\ref{cor:top-transform-sum} as
	$$\sum_{j=0}^{g}F[m\overline{L}](md-e_j)=\widehat{\mathrm{deg}}\left((\det \overline{\shE}_N)|_{Q_0}\right)+O(m).$$
	
	\paragraph{} We now compute the underlying line bundle $\det\shE_N$.
	
	\begin{lemma}\label{lem:picard-determinant}
		For $N\ge 2g-1$, the determinant line bundle satisfies
		\[
		\det\shE_N=\shO(-\theta_{x_0})\in\Pic(J).
		\]
	\end{lemma}
	
	\begin{proof}
		For any $N\in\mathbb{Z}$, we have an exact sequence
		$$0\to \shP_C+(N-1)p_C^*x_0\to\shP_C+Np_C^*x_0\to(\shP_C+Np_C^*x_0)|_{J\times x_0}\cong\shO_{J\times x_0}\to 0.$$
		Applying $\det\mathbf{R}p_{J*}$, we obtain
		$$\det\mathbf{R}p_{J*}(\shP_C+(N-1)p_C^*x_0)=\det\mathbf{R}p_{J*}(\shP_C+Np_C^*x_0).$$
		For $N\ge 2g-1$, by the theorem on cohomology and base change, all higher direct images vanish. Therefore,
		$$\det\mathbf{R}p_{J*}(\shP_C+Np_C^*x_0)=\det\shE_N.$$
		Now we set $N=g-1$. The identity
		$$\det\mathbf{R}p_{J*}(\shP_C+Np_C^*x_0)=\shO(-\theta_{x_0})$$
		follows from Proposition 2.4.2 and Corollary 2.7.10(ii) of [\hyperlink{Mor85}{Mor85}].
	\end{proof}
	
	\begin{proposition}\label{prop:degree-height}
		Let $N=2g$. As $Q_0\in J(K)$ varies with $\hat{h}(Q_0)\to\infty$, we have
		\[
		\widehat{\mathrm{deg}}\left((\det \overline{\shE}_N)|_{Q_0}\right)=-([K:\mathbb{Q}]+o(1))\hat{h}(Q_0)+O(1).
		\]
	\end{proposition}
	
	\begin{proof}
		Since $\theta_{x_0}$ and $[-1]^*\theta_{x_0}$ are algebraically equivalent, their associated Weil heights differ by $o(\hat{h}(Q_0))$. Therefore, it suffices to show that for a Weil height $h_1$ on $J$ associated to $\theta_{x_0}$,
		$$\widehat{\mathrm{deg}}\left((\det \overline{\shE}_N)|_{Q_0}\right)=-[K:\mathbb{Q}]h_1(Q_0)+O(1).$$
		The main problem is that the metric on $\overline{\shE}_N$ is not necessarily adelic. We now modify it to obtain an adelic metric.
		
		\paragraph{} Denote by $\mathcal{J}$ and $\mathcal{C}$ integral projective models of $J$ and $C$. Take $S_1\subset M_{K,f}$ finite so that $\shJ\times_{O_K}\shC$ is regular over $\Spec O_K\setminus S_1$. By the coherence condition, we may take a line bundle $\shP_C'\in\Pic((\shJ\times_{O_K}\shC)_{\mathrm{reg}})$ inducing $\overline{\shP}_C$ for all $v\in M_{K,f}\setminus(S_1\cup S_2)$, where $S_2$ is a finite set of finite places disjoint from $S_1$.
		
		\paragraph{} Set $\mathcal{U}=\Spec O_K\setminus(S_1\cup S_2)$. Denote by
		$$p_\shJ':\shJ_\mathcal{U}\times_\mathcal{U}\shC_\mathcal{U}\to \shJ_\mathcal{U}$$
		the restriction of $p_\shJ$, and $p_\shC'$ similarly. Extend the divisor $x_0$ on $C$ to a Cartier divisor $\mathcal{X}_0$ on $\shC_{\mathcal{U}}$. Let $\shE_N'$ be the coherent sheaf $$(p_\shJ')_*((\shP_C'+Np'^*_{\shC}\shO(\mathcal{X}_0))|_\mathcal{U}).$$
		
		\paragraph{} Take a neighborhood of the generic fiber $J\subset\shJ$, over which $\shE_N'$ is locally free and
		$$R^1(p_\shJ')_*((\shP_C'+Np'^*_{\shC}\shO(\mathcal{X}_0))|_\mathcal{U})$$
		vanishes. Take $S_3\subset M_{K,f}$ finite so that such a neighborhood contains the preimage of $\Spec O_K\setminus (S_1\cup S_2\cup S_3)$. Then for
		$$v\in M_{K,f}\setminus (S_1\cup S_2\cup S_3),$$
		the metric on $\overline{\shE}_N$ is induced by the model. Endow $\shE_N$ with any model metric for $v\in S_1\cup S_2\cup S_3$ and any smooth hermitian metric for $v\in M_{K,\infty}$. This construction gives rise to an adelic vector bundle $\overline{\shE}_N'$ extending $\shE_N$.
		
		\paragraph{} By a compactness argument analogous to that in Lemma~\ref{lem:sup-norm-bound}(2), we have
		$$\widehat{\mathrm{deg}}\left((\det \overline{\shE}_N)|_{Q_0}\right)=\widehat{\mathrm{deg}}\left((\det \overline{\shE}_N')|_{Q_0}\right)+O(1).$$
		Moreover, $Q_0\mapsto\widehat{\mathrm{deg}}\left((\det \overline{\shE}_N')|_{Q_0}\right)$ is $[K:\mathbb{Q}]$ times a Weil height associated to $\det\shE_N=\shO(-\theta_{x_0})$. This completes the proof.
	\end{proof}
	
	\paragraph{} Now we obtain the desired inequality:
	
	\begin{corollary}\label{cor:first-inequality}
		We have $D(\overline{L})\ge [K:\mathbb{Q}]\hat{h}(L-dx_0)$.
	\end{corollary}
	
	\begin{proof}
		It suffices to consider the case in which $\hat{h}(L-dx_0)\neq 0$. Combine Lemma~\ref{lem:hd-top-discrete-transforms}, Corollary~\ref{cor:top-transform-sum}, and Proposition~\ref{prop:degree-height}, together with the quadratic growth
		$$\hat{h}(mL_0)=m^2\hat{h}(L_0)$$
		of the N{\'e}ron-Tate height.
	\end{proof}
	
	\section{The second inequality}\label{sec:second-inequality}
	
	\paragraph{} In this section, we prove the converse inequality $D(\overline{L})\le [K:\mathbb{Q}]\hat{h}(L-dx_0)$, which completes the proof of Theorem~\ref{thm:main}. The proof strategy is rather different. Since the result is independent of adelic extension, it suffices to prove the result for a particular adelic extension, which turns out to be the admissible extension. Let $\overline{x}_0$ be any adelic extension of the divisor $x_0$.
	
	\begin{lemma}\label{lem:homogeneity-monotonicity}
		We have the following equality and inequality:
		\begin{enumerate}
			\item For $m>0$, $D(m\overline{L})=m^2D(\overline{L})$;
			\item $D(\overline{L}+\overline{x}_0)\le D(\overline{L})$.
		\end{enumerate}
	\end{lemma}
	
	\begin{proof}
		(1) is straightforward from the definition. For (2), we have
		$$F[m\overline{L}](\al)-F[m(\overline{L}+\overline{x}_0)](\al+m)\le -mF[\overline{x}_0](1).$$
		So the desired result follows from the upper-boundary effect in Remark~\ref{rmk:one-face}, where we take $n=1$.
	\end{proof}
	
	\begin{lemma}\label{lem:uniform-transform-bound}
		Let $\overline{L}_1, \overline{L}_2$ be two adelic line bundles on $C$. Then there exists a constant $c>0$ such that for $a_1, a_2\in\mathbb{Z}_{\ge 0}$, $m\in \mathbb{Z}_{>0}$, and $\al\in\nu(m\overline{L})$, where $\overline{L}=a_1\overline{L}_1+a_2\overline{L}_2$,
		
		$$\frac{1}{m}F[m\overline{L}](\al)\le c\max\{a_1,a_2\}.$$
	\end{lemma}
	
	\begin{proof}
		This is a variant of Lemma 2.4 in [\hyperlink{Yua09}{Yua09}]. Inspection of the proof shows that the resulting bound is at most a constant multiple of $\max\{a_1, a_2\}$.\\
		
		Applying the proof of [\hyperlink{Yua09}{Yua09}, Lem. 2.4(2a)] to a common model for $L_1$ and $L_2$, we obtain a finite set $S\subset M_K$, independent of $a_i$ and $m$, such that for any $v\in M_K\setminus S$, $\frac{1}{m}F_v[m\overline{L}]\le 0$. For $v\in S$, we want to show that $\frac{1}{m}F_v[m\overline{L}]\le c_v\max\{a_1,a_2\}$ for some constant $c_v$. This follows from the proofs of [\hyperlink{Nys14}{Nys14}, Lem. 5.4] and [\hyperlink{Yua09}{Yua09}, Lem. 2.3].
	\end{proof}
	
	\begin{proposition}\label{prop:second-inequality}
		We have $D(\overline{L})\le [K:\mathbb{Q}]\hat{h}(L-dx_0)$.
	\end{proposition}
	
	\begin{proof}
		By Lipman's desingularization theorem, we may take an integral, regular, flat, projective model $\pi:\mathcal{C}\to\Spec O_K$ of $C$. We claim that $\pi$ has geometrically connected fibers. Indeed, let $\mathcal{S}$ be the closure of $x_0$ in $\mathcal{C}$. Then $\mathcal{S}$ is integral and finite over $\Spec O_K$. Therefore, $\mathcal{S}$ is isomorphic to $\Spec O_K$. Since $H^0(C,\shO_C)=K$ and $O_K$ is integrally closed, the finite morphism appearing in the Stein factorization of $\pi$ is an isomorphism. Hence by Zariski's connectedness theorem, for any $y\in\Spec O_K$, $\shC_y$ is connected. Also, $\shC_y$ contains a $k(y)$-point $\mathcal{S}_y$; therefore, $\shC_y$ is geometrically connected.\\
		\par Extend $\mathcal{S}$ to an arithmetic divisor $\overline{\mathcal{S}}$ on $\shC$, such that its associated hermitian line bundle is Arakelov-admissible. We may assume that $\overline{x}_0$ is induced from $\overline{\mathcal{S}}$. Applying the arithmetic Hodge index theorem [\hyperlink{YG25}{YG25}, Thm. 5.4.2], there exists a $\mathbb{Q}$-hermitian line bundle $\overline{\shL}_0$ extending $L_0$ whose intersection with any vertical arithmetic divisor is $0$. By passing to a multiple and applying Lemma~\ref{lem:homogeneity-monotonicity}(1), we may assume that $\overline{\shL}_0$ is a hermitian line bundle and that $\overline{L}_0$ is induced by $\overline{\shL}_0$. We may assume without loss of generality that $\overline{L}=\overline{L}_0+d\shO(\overline{x}_0)$. By Lemma~\ref{lem:homogeneity-monotonicity}, we have
		\begin{equation}m^2D(\overline{L})\le D(m\overline{L}_0+\shO(\overline{x}_0)), \forall m\in\mathbb{Z}_{>0}.\end{equation}
		\par By Lemma~\ref{lem:uniform-transform-bound}, there is a constant $c$, independent of $m$, such that
		\begin{equation}\int_{\Delta(m\overline{L}_0+\shO(\overline{x}_0))}c[m\overline{L}_0+\shO(\overline{x}_0)](\al)d\al\le cm.\end{equation}
		\par Now fix $m\in\mathbb{Z}_{>0}$ and take $n\ge 2g$. Set $\overline{\shL}_1=n(m\overline{\shL}_0+\shO(\overline{\mathcal{S}}))$ and let $\overline{L}_1$ be its associated adelic line bundle. Since the intersection of $\overline{\shL}_0$ with an arithmetic divisor having only an archimedean part is $0$, the first Chern form of $\overline{\shL}_0$ must vanish; therefore, $\overline{\shL}_0$ is admissible. We now apply Faltings' version of the arithmetic Riemann–Roch theorem [\hyperlink{Fal84}{Fal84}, \S4, Thm. 3] on arithmetic surfaces. It says
		$$\chi(\overline{\shL}_1)=\frac{1}{2}\overline{\shL}_1\cdot(\overline{\shL}_1-\overline{\omega}_{\shC})+\chi(\overline{\shO}_{\shC}).$$
		Here, $\overline{\omega}_\shC$ is the relative canonical sheaf $\omega_\shC$ equipped with Arakelov's admissible metric. For a finitely generated $O_K$-module $M$ together with a Haar measure on $M\otimes_{\mathbb{Z}}\mathbb{R}$, define
		$$\chi(M)=-\log\frac{\mathrm{vol}(M\otimes_{\mathbb{Z}}\R/M)}{\# M_{\mathrm{tors}}}+\mathrm{rk}(M)\cdot\log\mathrm{vol}(O_K\otimes_{\mathbb{Z}}\R/O_K).$$
		A measure on $H^0(\shC,\overline{\shL}_1)\otimes_\mathbb{Z}\R$ is given in [\hyperlink{Fal84}{Fal84}], and
		$$\chi(\overline{\shL}_1)=\chi(H^0(\shC,\overline{\shL}_1))-\chi(H^1(\shC,\overline{\shL}_1)).$$
		Since $H^1(C, L_1)=0$ by Serre duality, $H^1(\shC,\overline{\shL}_1)$ is torsion. We have $\chi(H^1(\shC,\overline{\shL}_1))\ge 0$. Therefore, for $M=H^0(\shC,\overline{\shL}_1)$,
		\begin{equation}-\log\mathrm{vol}(M\otimes_{\mathbb{Z}}\R/M)\ge \frac{\left(m\overline{\shL}_0+\shO(\overline{\mathcal{S}})\right)^2}{2}n^2+O(n).\end{equation}
		Set $C_\C=C\times_\mathbb{Q}\C$, which is a disjoint union of Riemann surfaces. Fix a measure on $C_{\C}(\C)$ and denote by $B_{L^2}, B_{\sup}$ the unit balls with respect to the $L^2$-norm and the supremum norm. Applying [\hyperlink{Fal84}{Fal84}, \S3, Thm. 2] to all archimedean places, we obtain:
		$$\log\mathrm{vol}(B_{L^2}(H^0(C_\C,L_{1\C})))\ge-o(n^2).$$
		Since $M\otimes_{\mathbb{Z}}\R$ is a hermitian real form of $H^0(C_\C,L_{1\C})$, we have
		$$\log\mathrm{vol}(B_{L^2}(M\otimes_{\mathbb{Z}}\R))\ge \frac{1}{2}\log\mathrm{vol}(B_{L^2}(H^0(C_\C,L_{1\C})))\ge -o(n^2).$$ 
		Combining this with Gromov's inequality [\hyperlink{GS92}{GS92}, Lem. 30], which says that there exists a constant $A$ independent of $n$ such that for any $s\in H^0(C_{\C},L_{1\C})$, we have $\|s\|_{\sup}\le An\|s\|_{L^2}$, we obtain
		\begin{equation}\log\mathrm{vol}(B_{\sup}(M\otimes_{\mathbb{Z}}\R))\ge-o(n^2)-O(n\log n)=-o(n^2).\end{equation}
		Therefore, applying [\hyperlink{YG25}{YG25}, Prop. 8.5.6] and adding (3) and (4), we obtain
		$$\hat{\chi}(C,\overline{L}_1)=\hat{\chi}(\shC,\overline{\shL}_1)\ge\left(\frac{\left(m\overline{\shL}_0+\shO(\overline{\mathcal{S}})\right)^2}{2}-o(1)\right)n^2.$$
		Taking the limit as $n\to\infty$, we obtain $$\vol(m\overline{L}_0+\shO(\overline{x}_0))\ge \left(m\overline{\shL}_0+\shO(\overline{\mathcal{S}})\right)^2.$$ Combining this with (2), we obtain, as $m\to\infty$,
		$$D(m\overline{L}_0+\shO(\overline{x}_0))\le -\frac{1}{2}(\overline{\shL}_0)^2m^2+O(m).$$
		Combining this with (1), we obtain
		$$D(\overline{L})\le -\frac{1}{2}(\overline{\shL}_0)^2=[K:\mathbb{Q}]\hat{h}(L-dx_0)$$
		by the arithmetic Hodge index theorem.
	\end{proof}
	
	\section{The higher dimensional case}\label{sec:higher dim}
	
	\paragraph{} We conclude with a brief discussion of the case in which the projective variety $X$ has dimension $n\ge 2$. Fix a regular point $x_0\in X(K)$ and a system of parameters $t=(t_1,\cdots,t_n)$ at $x_0$. Assume that $\overline{L}$ is an adelic line bundle on $X$ such that $L$ is big. Fix a local generator $s_0$ of $L$ at $x_0$. Set $$D(\overline{L})=\int_{\Delta(L)}c[\overline{L}](\al)\d\al-\frac{1}{(n+1)!}\vol(\overline{L}).$$
	This quantity is non-negative by [\hyperlink{Yua09}{Yua09}]. In \S\ref{sec:hd-reduction}, we show that $D(\overline{L})$ is independent of the adelic metric. In \S\ref{sec:strict-ineq}, we prove the following lower bound for $D(\overline{L})$:
	
	\begin{theorem}\label{thm:ample-ineq}
		Assume that $X$ is smooth and that there is a flag of smooth subvarieties
		$$X=X_0\supset X_1\supset\cdots\supset X_{n-1}=C\supset X_n=\{x_0\}$$
		such that $\codim(X_{i+1},X_i)=1$, $t_i$ is a local equation of $X_i$ in $X_{i-1}$ at $x_0$ for any $1\le i\le n$, and that $C$ is a curve of genus $g\ge 1$.\\
		\par Let $\overline{L}$ be an adelic line bundle on $X$ with an ample underlying line bundle. Choose positive real numbers $\lambda_i\ (0\le i\le n-2)$ such that $\lambda_iL|_{X_i}-\shO_{X_i}(X_{i+1})$ is nef. Set
		$$v_i=\shO_{X_i}(X_{i+1})|_C-\deg\left(\shO_{X_i}(X_{i+1})|_C\right)x_0\in\Pic^0(C)$$
		and $u=L|_C-\deg(L|_C)x_0\in \Pic^0(C)$. Then,
		$$D(\overline{L})\ge\frac{1}{\lambda_0\cdots\lambda_{n-2}}[K:\mathbb{Q}]\int_{\Delta_{n-1}}\hat{h}_C\left(u-\sum_{i=0}^{n-2}\frac{\al_i}{\lambda_i}v_i\right) \d\al_0\cdots\d\al_{n-2}.$$
	\end{theorem}
	
	\paragraph{} Together with the toric case treated in [\hyperlink{Yua09}{Yua09}], Theorem~\ref{thm:ample-ineq} indicates that $D(\overline{L})$ depends on the system $t=(t_1,\cdots,t_n)$.
	
	\subsection{Jet-norm estimates}
	
	\paragraph{} In this subsection, we sketch a proof of the following jet-norm inequality, which can be seen as a relative version of [\hyperlink{Yua09}{Yua09}, Lem. 2.4]. The estimates will be useful for reducing higher-dimensional cases to the case of curves.
	
	\begin{lemma}\label{lem:jet-estimates}
		Let $X$ be a smooth projective variety over a number field $K$, and let $Y$ be a smooth subvariety of codimension $d$. Assume that there is a flag of smooth subvarieties
		$$X=X_0\supset X_1\supset\cdots\supset X_d=Y$$
		so that $\codim(X_{i+1},X_i)=1$. For $0\le i\le d-1$, denote by $\overline{M}_i$ an adelic extension of the line bundle $N_{X_{i+1}/X_i}|_{Y}$.\\
		\par Let $\overline{L}$ be an adelic line bundle on $X$. For any nonzero section $s\in H^0(X,mL)$, taking the expansion with respect to the flag yields a lexicographically least multi-index $$(\al_0,\cdots,\al_{d-1})\in(\mathbb{Z}_{\ge 0})^d$$
		and a canonical section $s^\circ\in H^0(Y,L^\circ)$, where $$\overline{L}^\circ=m\overline{L}|_Y-\sum_{i=0}^{d-1}\al_i\overline{M}_i.$$
		Then, there exist a constant $c_1$ and a finite set of places $S\subset M_K$, containing all archimedean places, such that for any $m$ and $s$,
		\begin{enumerate}
			\item For any $v\in M_K\setminus S$,
			$$\log\|s\|_{v,\sup}\ge\log\|s^\circ\|_{v,\sup}.$$
			\item For any $v\in S$,
			$$\log\|s\|_{v,\sup}\ge\log\|s^\circ\|_{v,\sup}-c_1\left(m+\sum_{i=0}^{d-1}\al_i+1\right).$$
		\end{enumerate}
	\end{lemma}
	
	\begin{proof}
		First assume that $d=1$. Extend $\shO_X(Y)$ to an adelic line bundle $\overline{\shO}_X(Y)$. We may assume that $\overline{M}_0$ has the metric induced via the natural line-bundle isomorphism $N_{Y/X}\cong\overline{\shO}_X(Y)|_Y$, since this will not affect the validity of the statement by Lemma~\ref{lem:trivial-underlying-divisor}. We show that for any place $v$, there is some constant $c_v$ such that
		$$\log\|s\|_{v,\sup}\ge\log\|s^\circ\|_{v,\sup}-c_v\left(m+\sum_{i=0}^{d-1}\al_i+1\right).$$
		For archimedean $v$, apply the Cauchy integral formula to holomorphic tubular neighborhoods of compact subsets of coordinate charts, and cover $Y$ with finitely many such subsets.\\
		\par For nonarchimedean $v$, let $R$ be the localization of $O_K$ at the prime associated to $v$. Take a flat projective model $\mathcal{X}$ of $X$ over $R$. Let $\mathcal{Y}$ be the closure of $Y$. We may assume that $\mathcal{X}$ is normal and that the metric on $\overline{L}$ at $v$ is induced by a line bundle $\shL$ on $\mathcal{X}$. Moreover, we may also assume that the metric on $\shO_X(Y)$ at $v$ is induced by a Cartier divisor $\mathcal{D}$ on $\mathcal{X}$ whose horizontal part is $\mathcal{Y}$ and whose vertical components have non-positive multiplicities. These modifications change the logarithm of the norm by $O(m+\al_0)$.\\
		\par After rescaling, we may assume that $\|s\|_{v,\sup}=1$. Therefore, $s$ lifts to a section $\tilde{s}\in H^0(\mathcal{X},m\shL)$. By definition, $s$ has order $\al_0$ along $Y$. Therefore, dividing $\tilde{s}$ by the defining rational section of $\shO_{\mathcal{X}}(\al_0\mathcal{D})$, we obtain a global section $\tilde{s}^\circ$ of $m\shL-\al_0\mathcal{D}$. Hence $s^\circ$ lifts to a section
		$$\tilde{s}^\circ|_{\mathcal{Y}}\in H^0(\mathcal{Y},(m\shL-\al_0\mathcal{D})|_\mathcal{Y}).$$
		Therefore, $\|s^\circ\|_{v,\sup}\le 1$.\\
		\par It remains to show that there is a finite set $S\subset M_{K,f}$, independent of $m$, such that for any finite place $v\notin S$, we have $\|s\|_{v,\sup}\ge\|s^\circ\|_{v,\sup}$. This follows from the preceding argument by allowing $v$ to vary and taking $\mathcal{X}$ to be obtained by base change from a fixed model of $X$ over $\Spec O_K$.\\
		\par We may generalize the above argument, replacing $m\overline{L}$ by $$m_1\overline{L}_1+\cdots+m_k\overline{L}_k,$$
		and replacing the error term $c_1(m+\sum_{i=0}^{d-1}\al_i+1)$ by $c_1(\sum_{i=1}^k|m_i|+\sum_{i=0}^{d-1}\al_i+1)$. Now for general $d$, we proceed by induction on $d$. Each time, we pass to a successive jet
		$$s^\circ_i\in H^0\left(mL|_{X_i}-\sum_{j=0}^{i-1}\al_jN_{X_{j+1}/X_j}|_{X_i}\right).$$
	\end{proof}
	
	\subsection{Strict inequality in higher dimensions}\label{sec:strict-ineq}
	
	\paragraph{} We now prove Theorem~\ref{thm:ample-ineq} using the jet-norm estimates and results for curves.
	
	\begin{proof}
		After approximating the $\lambda_i$ by rationals, we may assume that $\lambda_i\in\mathbb{Q}$. Choose $N$ sufficiently large. Denote by $\delta_{ij}$ the Kronecker delta. Then, for any integers $\al_i\ge 0$ and $m\ge\sum_{i=0}^{n-2}\lambda_i\al_i+N$, the line bundles
		$$\lceil\lambda_j(\al_j+\delta_{ij})\rceil L|_{X_i}-(\al_j+\delta_{ij})\shO_{X_j}(X_{j+1})|_{X_i}\ (i\ge j)$$
		are nef, and the line bundle
		$$\left(m-\sum_{j=0}^i\lceil\lambda_j(\al_j+\delta_{ij})\rceil\right)L|_{X_i}-K_{X_i}$$
		is ample. The Kodaira vanishing theorem gives:
		$$H^1\left(X_i,mL|_{X_i}-\sum_{j=0}^i(\al_j+\delta_{ij})\shO_{X_j}(X_{j+1})|_{X_i}\right)=0.$$
		Therefore, the natural map
		$$H^0\left(X_i,mL|_{X_i}-\sum_{j=0}^i\al_j\shO_{X_j}(X_{j+1})|_{X_i}\right)\to H^0\left(X_{i+1},mL|_{X_{i+1}}-\sum_{j=0}^i\al_j\shO_{X_j}(X_{j+1})|_{X_{i+1}}\right)$$
		is surjective. Consequently, $(\al_0,\cdots,\al_{n-2},\al_{n-1})\in \nu(mL)$ if and only if
		$$\al_{n-1}\in \nu\left(mL|_C-\sum_{j=0}^{n-2}\al_j\shO_{X_j}(X_{j+1})|_C\right).$$
		By Riemann–Roch, the above condition holds for any
		$$0\le \al_{n-1}\le m\deg(L|_C)-\sum_{j=0}^{n-2}\al_j\deg\left(\shO_{X_j}(X_{j+1})|_C\right)-2g$$
		and fails for any
		$$\al_{n-1}> m\deg(L|_C)-\sum_{j=0}^{n-2}\al_j\deg\left(\shO_{X_j}(X_{j+1})|_C\right).$$
		Passing to the closure implies that for any nonnegative real numbers $\be_0,\cdots,\be_{n-1}$, if $\sum_{i=0}^{n-2}\lambda_i\be_i\le 1$, then $(\be_0,\cdots,\be_{n-1})\in\Delta(L)$ if and only if
		$$\sum_{j=0}^{n-2}\be_j\deg\left(\shO_{X_j}(X_{j+1})|_C\right)+\be_{n-1}\le \deg(L|_C).$$
		\par Equip each $\shO_{X_i}(X_{i+1})$ with an adelic extension $\overline{\shO}_{X_i}(X_{i+1})$, and equip the restricted normal bundle $M_i=N_{X_{i+1}/X_i}|_C$ with induced adelic extension $\overline{M}_i$. By the jet-norm estimates (Lemma~\ref{lem:jet-estimates}), there is a constant $c_1$ such that for any $m$ and
		$$(\al_0,\cdots,\al_{n-2},\al_{n-1})\in \nu(mL),$$
		we have:
		$$F[m\overline{L}](\al_0,\cdots,\al_{n-1})\le F\left[m\overline{L}|_C-\sum_{i=0}^{n-2}\al_i\overline{M}_i\right](\al_{n-1})+c_1m.$$
		Now restrict attention to the case in which $m\ge\sum_{i=0}^{n-2}\lambda_i\al_i+N$ and
		$$\al_{n-1}\ge\deg\left(mL|_C-\sum_{i=0}^{n-2}\al_iM_i\right)-2g\ge N-2g\ge 0$$
		By our previous description of $\Delta(L)$, such points contribute to the boundary term in Lemma~\ref{lem:hd-top-discrete-transforms}. Moreover, the arguments of \S\S\ref{sec:highest-transforms} and \ref{sec:picard-determinant} give the following uniform estimate:\\
		
		\par \textit{Given a curve $C$ as in \S\ref{sec:first-inequality} and adelic line bundles $\overline{L}_i$ over $C$ $(1\le i\le k)$. For any integers $m_i\ (1\le i\le k)$ with $\sum_{i=1}^km_i\deg L_i\ge 2g$, set $\overline{H}=\sum_{i=1}^km_i\overline{L}_i$. We have}
		$$\sum_{\al\in\nu(H), \al\ge\deg(H)-2g}F[\overline{H}](\al)=-([K:\mathbb{Q}]+o(1))\hat{h}\left(H-(\deg H)x_0\right)+O\left(\sum_{i=1}^k|m_i|+1\right).$$
		\textit{Here, the $o(1)$-term tends to $0$ uniformly as $\hat{h}_C\left(mu-\sum_{i=0}^{n-2}\al_iv_i\right)\to\infty$, and its absolute value may be bounded by $1$.\\}
		
		\par Therefore, for $\al_0,\cdots,\al_{n-2}$ given, the sum of $F[m\overline{L}](\al_0,\cdots,\al_{n-1})$ over all such $\al_{n-1}$ is at most
		$$(-[K:\mathbb{Q}]+o(1))\hat{h}_C\left(mu-\sum_{i=0}^{n-2}\al_iv_i\right)+O(m).$$
		For any $\eps_1>0$, applying Lemma~\ref{lem:hd-top-discrete-transforms}, we obtain
		\begin{equation}\notag
			\begin{split}
				D(\overline{L}) &\ \ge\lim_{\eps\to 0^+}\lim_{m\to\infty}\frac{1}{m^{n+1}}\sum_{(*)}([K:\mathbb{Q}]-\eps_1)\hat{h}_C\left(mu-\sum_{i=0}^{n-2}\al_iv_i\right)\\
				&\ =\lim_{\eps\to 0^+}\lim_{m\to\infty}\frac{1}{m^{n-1}}\sum_{(*)}([K:\mathbb{Q}]-\eps_1)\hat{h}_C\left(u-\sum_{i=0}^{n-2}\frac{\al_i}{m}v_i\right)\\
				&\ =\frac{1}{\lambda_0\cdots\lambda_{n-2}}([K:\mathbb{Q}]-\eps_1)\int_{\Delta_{n-1}}\hat{h}_C\left(u-\sum_{i=0}^{n-2}\frac{\al_i}{\lambda_i}v_i\right) \d\al_0\cdots\d\al_{n-2}.
			\end{split}
		\end{equation}
		Here, the Riemann sum $(*)$ is taken over tuples $(\al_0,\cdots,\al_{n-2})$ satisfying $\al_i\in\mathbb{Z}_{\ge 0}$ and $m\ge\sum_{i=0}^{n-2}\lambda_i\al_i+N$. Letting $\eps_1\to 0^+$ completes the proof.
	\end{proof}
	
	\paragraph{} In the case $X=\P_K^2$ and $\overline{L}$ is an extension of $L=\shO(1)$, choose $t=(t_1,t_2)$ so that the local zero locus $\{t_1=0\}$ is a nonsingular cubic curve $C$, and $t_2$ is a parameter of $C$ at $x_0$. Then,
	$$D(\overline{L})\ge\frac{1}{9}D(\overline{L}|_C)=\frac{1}{9}[K:\mathbb{Q}]\hat{h}_C(L|_C-3x_0).$$
	If $L|_C-3x_0$ is non-torsion, then the right-hand side is positive. On the other hand, by [\hyperlink{Yua09}{Yua09}], for a parameter system $t'=(t_1',t_2')$ for which each $\{t_i'=0\}$ is a line, we have $D(\overline{L})=0$. Therefore, the quantity $D(\overline{L})$ depends on the choice of parameters for the Okounkov body.
	
	\paragraph{} As another special case, consider $X=C\times\P_K^{n-1}\ (n\ge 2)$ and the datum in Remark~\ref{rmk:one-face}. We have the following formula:
	
	\begin{proposition}\label{prop:product-formula}
		Denote by $\overline{M}$ any adelic extension of $M$. Then,
		$$D(\overline{L})=\frac{b^{n-1}}{(n-1)!}D(\overline{M})=\frac{b^{n-1}}{(n-1)!}[K:\mathbb{Q}]\hat{h}_C(M-dy_0).$$
	\end{proposition}
	
	\begin{proof}
		We have $\nu(mL)=\nu(\shO(bm))\times\nu(mM)$. Given adelic extensions $\overline{M}$ and $\overline{\shO}(1)$, we may assume that $\overline{L}=\overline{M}\boxtimes\overline{\shO}(b)$ by metric-independence. By Lemma~\ref{lem:jet-estimates}, there is some constant $c_1$ such that for any $m$ and $(\al,\be)\in\nu(\shO(bm))\times\nu(mM)$, we have
		$$F[m\overline{L}](\al,\be)\le F[m\overline{M}](\be)+c_1m.$$
		Conversely, by considering product sections, one finds a constant $c_2$ such that
		$$F[m\overline{L}](\al,\be)\ge F[m\overline{M}](\be)+c_2m.$$
		By Remark~\ref{rmk:one-face}, we have
		\begin{equation}\notag
			\begin{split}
				D(\overline{L})&\ =\lim_{\eps\to 0^+}\lim_{m\to\infty}\frac{1}{m^{n+1}}\sum_{(\al,\be)\in\nu(mL),\be/m>d-\eps}\left(-F[m\overline{L}](\al,\be)\right)\\
				&\ =\lim_{\eps\to 0^+}\lim_{m\to\infty}\frac{1}{m^{n+1}}\sum_{(\al,\be)\in\nu(mL),\be/m>d-\eps}\left(-F[m\overline{M}](\be)\right)\\
				&\ =\lim_{\eps\to 0^+}\lim_{m\to\infty}\frac{1}{m^{n+1}}\binom{bm+n-1}{n-1}\sum_{\be\in\nu(mM),\be/m>d-\eps}\left(-F[m\overline{M}](\be)\right)\\
				&\ =\frac{b^{n-1}}{(n-1)!}D(\overline{M}).
			\end{split}
		\end{equation}
	\end{proof}
	
	\appendix
	
	\section{Comparison with other transforms}\label{app:comparison}
	
	\paragraph{} In this appendix, we list several historical constructions of similar concave (or convex) transforms.
	
	\subsection*{Arithmetic capacity theory}
	
	\paragraph{} The arithmetic capacity theory developed by Cantor, Rumely, Lau, Chinburg, and Varley involves a convex transform for each place $v\in M_K$. In [\hyperlink{RL94}{RL94}], for an ample divisor $D$ on $\P_K^N$, a place $v\in M_K$, and a subset $E_v\subset \P^N(\C_v)$ satisfying certain conditions, Rumely and Lau constructed a log-convex function
	$$\Che_v(\theta)=\Che_v(\theta; E_v, D):\Sigma_D\to\mathbb{R}_{\ge 0},$$
	where
	$$\Sigma_D=\{\theta=(\theta_0,\cdots,\theta_N)\in\mathbb{R}_{\ge 0}^{N+1}:\theta_0+\cdots+\theta_N=\deg D\}.$$
	They proved an integration formula ([\hyperlink{RL94}{RL94}, \S 2.A and Thm. 2.3]) relating the integral of $\log \Che$ to the \textit{local sectional capacity} $S_\gamma(E_v,D)$:
	$$\int_{\Sigma_D}\log\Che_v(\theta;E_v;D)\d\theta_1\cdots\d\theta_N=\frac{1}{(N+1)!}\log S_\gamma(E_v,D).$$
	Moreover, under certain conditions (A1)-(A3) on $\mathbb{E}=\{E_v\}_v$, they defined the \textit{global sectional capacity} $S_\gamma(\mathbb{E},D)$ in an analogous way with $\exp(-\vol)$, and showed that the product of local sectional capacities converges to it. Their construction is an adelic-set precursor of the Chebyshev-transform constructions later on.
	
	\paragraph{} Rumely and Lau's construction is related to Yuan's as follows. Choose the lexicographic order in [\hyperlink{RL94}{RL94}] to be the inverse of the usual one. Assume that $X=\P_K^n$, and let $\overline{L}$ be an adelic extension of $\shO(d)\ (d>0)$. Take $x_0=[0:\cdots:0:1]$, and let $t$ be the standard parameter system at $x_0$. We have a natural decomposition
	$$\P_K^{n+1}=\mathbb{A}_K^{n+1}\amalg H_\infty.$$
	Take $N=n+1$ and $D$ be the divisor $dH_\infty$ on $\P_K^N$. For $v\in M_K$, take
	$$E_v=\{z\in \mathbb{A}^{n+1}(\C_v)=\mathrm{Cone}(\P^n(\C_v)): z\neq 0, \|z^d\|_{-\overline{L},v}\le 1\}.$$
	Here, $z$ is viewed as a vector of $\shO_{X_{\C_v}}(-1)$ in the fiber over $[z]$. Then, the family $\{E_v\}_v$ satisfy (A1)-(A3) by the coherence condition. For $(\al_1,\cdots,\al_n)\in\Delta_t(L)^\circ$, we have
	$$n_vc_v[\overline{L}](\al_1,\cdots,\al_n)=-\log\Che_v(0,\al_1,\cdots,\al_n,d-\al_1-\cdots-\al_n; E_v, D).$$
	
	\subsection*{Toric varieties}
	
	\paragraph{} In the case of toric varieties, the function $c[\overline{L}]$ has been studied in the papers [\hyperlink{Mor11}{Mor11}] and [\hyperlink{BPS15}{BPS15}]. In the setting of [\hyperlink{Mor11}{Mor11}], let $K=\mathbb{Q}$, $X=\P_\mathbb{Q}^n$, $x_0=[1:0:\cdots:0]$, $t_i=\frac{X_i}{X_0}$, and $\overline{L}$ be induced by the arithmetic divisor $\overline{D}_a$. Then, $c[\overline{L}]$ is related to the function $\varphi_a$ by
	$$c[\overline{L}](\al_1,\cdots,\al_n)=\frac{1}{2}\varphi_a(1-\al_1-\cdots-\al_n,\al_1,\cdots,\al_n)$$
	where $\al_i\in\R_{>0}$ and $\al_1+\cdots+\al_n<1$. Indeed, the roots-of-unity averaging argument in [\hyperlink{Mor11}{Mor11}, Claim 1.5.1] shows that the section $T_0^{e_0}\cdots T_n^{e_n}$ minimize the $\sup$ norm among all the sections in which the coefficient of this monomial is $1$. Its $\sup$ norm is explicitly computed in [\hyperlink{Mor11}{Mor11}, Prop. 1.3(1)].
	
	\paragraph{} In the setting of [\hyperlink{BPS15}{BPS15}], let $X$ be a projective toric variety associated to a toric datum $\Sigma$. Take a smooth maximal cone $\sigma$ in $\Sigma$ and a system of generators $v_i\ (1\le i\le n)$ for $\sigma$. Then, we obtain a torus-invariant regular point $x_0=x_\sigma\in X(K)$ and a system of parameters $t$ at $x_0$ consisting of torus eigenfunctions.
	
	\paragraph{} Let $\overline{L}$ be induced by a toric metrized divisor $\overline{D}$. The domain $\Delta_D\subset M_\R$ of $\vartheta_{\overline{D}}$ is identified with $\Delta_t(L)$ via
	$$m\mapsto(\langle m-m_\sigma,v_1\rangle,\cdots,\langle m-m_\sigma,v_n\rangle),$$
	where $m_\sigma\in M_\R$ is the Cartier datum of $D$ along $\sigma$. A similar averaging argument implies that
	$$c[\overline{L}](\al)=[K:\mathbb{Q}]\vartheta_{\overline{D}}(\al),\  \forall\al\in\Delta_t(L)^\circ.$$
	
	\subsection*{Boucksom and Chen's construction}
	
	\paragraph{} In [\hyperlink{BC11}{BC11}], Boucksom and Chen constructed another concave transform $G_{\overline{L}}$ on $\Delta_t(L)$. They showed that their transform is different from Yuan's in general, but agrees in the toric case in the sense that $c[\overline{L}]=[K:\mathbb{Q}]G_{\overline{L}}$. When $L$ is big and $\overline{L}$ is nef, they proved the identity
	$$\int_{\Delta_t(L)}G_{\overline{L}}(\al)\d\al=\frac{1}{[K:\mathbb{Q}](n+1)!}\vol(\overline{L}).$$
	In this case, Balla{\"y} showed in [\hyperlink{Bal21}{Bal21}] and [\hyperlink{Bal24}{Bal24}] respectively that the supremum and infimum of $G_{\overline{L}}$ are the essential and absolute minima of $\overline{L}$. This construction generalizes to adelic line bundles on quasi-projective varieties, see [\hyperlink{BC25}{BC25}].
	
	\section*{References}
	\addcontentsline{toc}{section}{References}
	
	\paragraph{} \hypertarget{Bal21}{[Bal21]} F. Balla{\"y}, \textit{Successive minima and asymptotic slopes in Arakelov geometry}. Compos. Math., 157, no. 6, 1302–1339. 2021.
	
	\paragraph{} \hypertarget{Bal24}{[Bal24]} F. Balla{\"y}, \textit{Arithmetic Okounkov bodies and positivity of adelic Cartier divisors}. J. Algebraic Geom., 33, no. 3, 455–492. 2024.
	
	\paragraph{} \hypertarget{BC11}{[BC11]} S. Boucksom and H. Chen, \textit{Okounkov bodies of filtered linear series}.
	Compos. Math.147 , no. 4, 1205–1229. 2011.
	
	\paragraph{} \hypertarget{BC25}{[BC25]} D. Biswas and Y. Cai, \textit{Concave transforms of compactified $S$-metrized divisors}. arXiv:2505.14023. 2025.
	
	\paragraph{} \hypertarget{BPS15}{[BPS15]} J. I. Burgos Gil, P. Philippon, and M. Sombra, \textit{Successive minima of toric height functions}. Ann. Inst. Fourier (Grenoble), 65, no. 5, 2145–2197. 2015.
	
	\paragraph{} \hypertarget{Fal84}{[Fal84]} G. Faltings, \textit{Calculus on arithmetic surfaces}. Ann. of Math. (2),
	119, no. 2, 387–424. 1984.
	
	\paragraph{} \hypertarget{GS92}{[GS92]} H. Gillet and C. Soul{\'e}, \textit{An arithmetic Riemann–Roch theorem}. Invent. Math., 110, 473–543. 1992.
	
	\paragraph{} \hypertarget{Mor11}{[Mor11]} A. Moriwaki, \textit{Big arithmetic divisors on the projective spaces over $\mathbb{Z}$}. Kyoto J. Math., 51, no. 3, 503–534. 2011.
	
	\paragraph{} \hypertarget{Mor85}{[Mor85]} L. Moret-Bailly, \textit{M{\'e}triques permises}. Ast{\'e}risque 127, 29–87. 1985.
	
	\paragraph{} \hypertarget{Nys14}{[Nys14]} D. Witt Nystr{\" o}m, \textit{Transforming metrics on a line bundle to the Okounkov body}. Ann. Sci. {\'E}c. Norm. Sup{\'e}r. (4), 47, no. 6, 1111–1161. 2014.
	
	\paragraph{} \hypertarget{RL94}{[RL94]} R. Rumely and C. F. Lau, \textit{Arithmetic capacities on $\P^N$}. Math. Z., 215, no. 4, 533–560. 1994.
	
	\paragraph{} \hypertarget{Yua09}{[Yua09]} X. Yuan, \textit{On Volumes of Arithmetic Line Bundles II}. Preprint. arXiv:0909.3680 [math.AG]. 2009.
	
	\paragraph{} \hypertarget{YG25}{[YG25]} X. Yuan and R. Guo, \textit{A First Course in Arakelov Geometry}. Preprint. 2025.
	
	\paragraph{} \hypertarget{YZ26}{[YZ26]} X. Yuan and S. Zhang, \textit{Adelic line bundles on quasi-projective varieties}. Annals of Mathematics Studies, vol. 221. Princeton University Press, Princeton. 2026. DOI: 10.1515/9780691278704.\\
	
	\
	
	{\footnotesize
		Address: School of Mathematical Sciences, Peking University, No. 5 Yiheyuan Road, Haidian District, Beijing, 100871, P.R. China
		
		Email: yxye25@stu.pku.edu.cn
	}

\end{document}